\documentclass[12pt]{amsart}
\def\data{24 January 2020}
\usepackage{amssymb}
\usepackage{amsfonts}
\usepackage{amsmath}
\usepackage{amsthm}
\usepackage[pdftex,bookmarks,colorlinks]{hyperref}
\usepackage[pdftex]{hyperref}
\usepackage{url}
\usepackage{color}
\usepackage{graphicx}

\newcommand{\bq}{\begin{quote}}
\newcommand{\eq}{\end{quote}}
\newcommand{\bi}{\begin{itemize}}
\newcommand{\ei}{\end{itemize}}
\newcommand{\bd}{\begin{description}}
\newcommand{\ed}{\end{description}}
\newcommand{\ben}{\begin{enumerate}}
\newcommand{\een}{\end{enumerate}}
\newcommand{\bbm}{\begin{bmatrix}}
\newcommand{\ebm}{\end{bmatrix}}
\newcommand{\bea}{\begin{eqnarray*}}
\newcommand{\eea}{\end{eqnarray*}}

\newtheorem{theorem}{Theorem}
\newtheorem{lemma}{Lemma}[section]
\newtheorem{proposition}[lemma]{Proposition}

\newtheorem{remark}[theorem]{Remark}
\newtheorem{corollary}[theorem]{Corollary}

\def\bF{\mathbb{F}}

\def\lsl{\mathfrak{sl}}

\def\KK{\mathsf{K}}

\def\QQ{\mathsf{Q}}
\def\RR{\mathsf{R}}
\def\SS{\mathsf{S}}

\def\XX{\mathsf{X}}
\def\YY{\mathsf{Y}}

\newcommand{\bx}[1]{\mathsf{#1}}     
\newcommand{\bxg}[1]{\boldsymbol #1} 

\def\Sym{{\rm Sym}}
\def\2G2{\ensuremath{^2{\rm G}_2}}
\def\sl{{\rm{SL}}}
\def\gl{{\rm{GL}}}

\def\psl{{\rm{PSL}}}
\def\pgl{{\rm{PGL}}}

\def\so{{\rm{SO}}}

\def\om{{\rm{\Omega }}}

\definecolor{darkgreen}{rgb}{0,0.6,0}

\newcommand{\encr}{\ensuremath{\vDash}}

\AtBeginDocument{%
  \def\MR#1{}
  }

\begin{document}

\title[Natural representations of black box groups $\sl_2(\bF_q)$]{Natural representations of black box groups encrypting $\sl_2(\bF_q)$}

\author{Alexandre Borovik}
\address{Department of Mathematics, University of Manchester, UK}
\email{alexandre@borovik.net}
\author{\c{S}\"{u}kr\"{u} Yal\c{c}\i nkaya}
\address{Department of Mathematics, Istanbul University, Turkey}
\email{sukru.yalcinkaya@istanbul.edu.tr}

\date{\data}

\begin{abstract}
Given a global exponent $E$ for a black box group $\YY$ encrypting $\sl_2(\bF)$, where $\bF$ is an unknown finite field of unknown odd characteristic, we construct, in probabilistic time polynomial in $\log E$, the isomorphisms
\[
\YY \longleftrightarrow \sl_2(\KK),
\]
 where $\KK$ is a black box field encrypting $\bF$. Our algorithm makes no reference to any additional oracles. We also give similar algorithms for black box groups encrypting $\pgl_2(\bF)$, $\psl_2(\bF)$.

\end{abstract}

\maketitle


\section{Introduction}

The present paper extends the results of our previous paper \cite{BY2018} and uses its notation and terminology. In \cite{BY2018}, we presented an algorithm constructing the adjoint representation of a black box group $\YY$ encrypting $\psl_2(\bF)$ for a field $\bF$ of odd order, that is, we constructed a black box field $\KK$ encrypting $\bF$ in $\YY$ and represented elements of $\YY$ as $3 \times 3$ orthogonal matrices with entries from this black box field $\KK$.

In this paper, we use the same setup as in \cite{BY2018} to produce an algorithm which constructs the natural representation of a black box group $\XX$ encrypting $\sl_2(\bF)$. The need for such a construction arises from the fact that a constructive recognition algorithm for black box groups of Lie type of high rank  involves a constructive recognition of a black box group encrypting $\sl_2(\bF)$ \cite{brooksbank03.162,brooksbank08.885, brooksbank01.95,brooksbank06.256}, where the latter is known as the $\sl_2$-oracle. These papers are based on the use of the $\sl_2(\bF)$-oracle as well as the discrete logarithm oracle in $\bF$. We wish to emphasise that we make no use of any oracles.



We prove the following theorem.

\begin{theorem}\label{theo:proxy-sl2}
Let\/ $\YY$ be a black box group encrypting $\sl_2(\bF)$, where $\bF$ is an unknown finite field of unknown odd characteristic, and\/ $E$ be a global exponent for\/ $\YY$. Then there is a Las Vegas algorithm which constructs, in probabilistic time polynomial in $\log E$,

\begin{enumerate}
\item[(a)] a black box field $\KK$ encrypting $\bF$, and
\item[(b)] isomorphisms
\[
\bxg{\psi}: \YY \longrightarrow \sl_2(\KK) \mbox{ and } \bxg{\psi}^{-1}: \sl_2(\KK) \longrightarrow \YY
\]
which run in probabilistic time polynomial in $\log E$.
\end{enumerate}
\end{theorem}

In addition, we call the pair of isomorphisms $\bxg{\psi}$ and $\bxg{\psi}^{-1}$  constructed in Theorem \ref{theo:proxy-sl2} a \emph{structural approximation of} $\YY$ and the matrix group $\sl_2(\KK)$ its \emph{structural proxy}.

We also write $\XX\encr G$ when a black box group (or a ring, field, etc.) $\XX$ encrypts a group (ring, field, etc.) $G$.

Ignoring finer points of the algebraic group theory which are less relevant for finite groups in the black box setup, most groups of Lie type (to avoid technical details, we exclude series $^2B_2$, $^2F_4$, and $^2G_2$) can be seen as functors $G: \mathcal{R} \longrightarrow \mathcal{G}$ from the category of commutative unital rings $\mathcal{R}$ with involution (that is, an automorphism of order $\leq 2$)  to the category of groups. There are other algebraic structures which can be defined in a similar functorial way, as functors $A: \mathcal{R} \longrightarrow \mathcal{A}$ for example, finite dimensional associative algebras and finite dimensional Lie algebras viewed as rings. The corresponding structural proxy problem can be stated as follows.

\bi
\item \emph{Construction of a structural proxy}.  Suppose that we are given  a black box structure $\XX \encr A(\bF)$. Construct, in probabilistic polynomial  in $l(\XX)$ time, where $l(\XX)$ is the length of the strings in $\XX$,
    \bi
     \item a black box field $\KK\encr \bF$, and
     \item probabilistic polynomial time isomorphisms $$\bxg{\psi}: A(\KK) \longrightarrow \XX$$
     and
     $$\bxg{\psi}^{-1}: \XX  \longrightarrow A(\KK).$$
     \ei
\ei


\begin{remark} \label{rem:so3} The key technical result in \cite{BY2018} amounts to construction of a structural proxy for the group $\XX \encr \so_3(\bF)\simeq \pgl_2(\bF)$,
\[
\XX \longleftrightarrow \so_3(\KK).
\]
In this paper we have another structural proxy for $\XX$,
\[
\XX \longleftrightarrow \pgl_2(\KK)
\]
with $\KK$ being the same black box field. Indeed, in Section \ref{sec:Corollary-2}, we construct an efficient isomorphism between $\so_3(\KK)$ and $\pgl_2(\KK)$. The existence of this isomorphism is well-known but its efficient computational realization needs some delicate treatment.
\end{remark}

\begin{theorem}\label{corollary}
Let $\XX$ be a black box group encrypting $\psl_2(\bF)$ or $\pgl_2(\bF)$,  where $\bF$ is an unknown finite field of unknown odd characteristic, and let $E$ be a global exponent for $\XX$. Then there is a Las Vegas algorithm which constructs, in probabilistic time polynomial in $\log E$, structural proxies \[
\XX \longleftrightarrow \psl_2(\KK)
\]
or
\[
\XX \longleftrightarrow \pgl_2(\KK),
\]
respectively.
\end{theorem}

\begin{corollary}\label{corollary2}
Let\/ $\XX$ be a black box group encrypting one of the groups $\sl_2(\bF), \psl_2(\bF)$ or $\pgl_2(\bF)$, where $\bF$ is the standard explicitly given finite field of known characteristic. Then there is a Las Vegas algorithm which constructs, in probabilistic time polynomial in $\log |\bF|$, an isomorphism
\[
\sl_2(\bF) \rightarrow \XX, \, \psl_2(\bF) \rightarrow \XX \mbox{ or }\pgl_2(\bF) \rightarrow \XX,
\]
respectively. The running time of our algorithm is in probabilistic polynomial time in $\log |\bF|$.
\end{corollary}

The proof of Corollary \ref{corollary2} is achieved by constructing an isomorphism from $\bF$ to $\KK$ and such construction is given in \cite{maurer07.427}. We shall note here that the isomorphism $\sl_2(\bF) \rightarrow \XX$ in Corollary \ref{corollary2} is a half $\sl_2$-oracle. Moreover, when the characteristic of $\bF$ is small, we can reverse the isomorphism from $\bF$ to $\KK$ by using the the results in \cite{maurer07.427} which gives us a full $\sl_2$-oracle, that is, two way isomorphism between $\sl_2(\bF)$ and $\XX$.

Notice that if $|\bF|< 7$ then all our theorems are obviously true and do not require methods developed in this paper, since we can list elements of $\YY$ and provide necessary isomorphisms with no difficulty in these cases. In technical results in this paper, we assume that $|\bF| \geqslant 7$. We note that all classical black box groups over small fields can be recognized efficiently by the algorithms in \cite{kantor01.168}.

 \section{Plan of proof of Theorem \ref{theo:proxy-sl2}}\label{sec:plan}

In the proof of Theorem \ref{theo:proxy-sl2}, we first redefine the equality of strings in $\YY \encr \sl_2(\bF)$ in the following way to be able to pass to the quotient group $\YY/Z(\YY)$:
\[
\bx{x} \equiv \bx{y} \iff \bx{xy^{-1}} \mbox{ is either identity or the central involution in } \YY.
\]
Then, we use algorithms developed in  \cite[Theorem 1.3]{BY2018} for the black box group $\YY/Z(\YY)$ to construct a black box field $\KK\encr \bF$, a black box group $\XX \encr \pgl_2(\bF)$, and computable, in polynomial time, homomorphisms
\[
\YY \longrightarrow \YY/Z(\YY) \longrightarrow \XX \longleftrightarrow \so_3(\KK).
\]
We note here that we can add and multiply elements and take additive and multiplicative inverses of (nonzero) elements of the black box field $\KK$ (see \cite[Section 9]{BY2018}), and also the two isomorphisms
\[
\XX \longleftrightarrow \so_3(\KK)
\]
are inverses of each other (check \cite[Section 11]{BY2018}).

It is well-known that $\so_3(\KK)$ is isomorphic to $\pgl_2(\KK)$ and we need to present an efficient algorithm constructing such an isomorphism. We deal with this problem in Section \ref{sec:sl2-proxy}.

The group $\so_3(\KK)$ arises in \cite{BY2018}  as the group of matrices from  $\gl_3(\KK)$ preserving the quadratic form with the matrix
 \[
 \bbm 1&0&0\\ 0&1&0 \\ 0&0&1\ebm;
 \]
we will denote this group as $\so_3^{\sharp}(\KK)$.

It turns out that it is much more convenient to compute in the orthogonal groups $\so_3(\KK)$ preserving the quadratic form with the matrix
 \[
 \bbm 0&0&1\\ 0&-2&0 \\ 1&0&0\ebm;
 \]
we will denote this group as $\so_3^{\flat}(\KK)$.

Of course the two groups $\so_3^{\sharp}(\KK)$ and $\so_3^{\flat}(\KK)$ are conjugate in $\gl_3(\KK)$; the computation of the conjugating (change of basis) matrix is very easy if $\KK$ contains $\sqrt{-1}$, but requires more attention if  $\KK$ does not contain $\sqrt{-1}$, see Section \ref{sec:change-of-basis}.

After that we get constructive homomorphisms
\[
 \YY \longrightarrow \YY/Z(\YY) \longrightarrow \XX \longleftrightarrow \so_3^\sharp(\KK) \longleftrightarrow \so_3^\flat (\KK) \longleftrightarrow \pgl_2(\KK)
\]
and focus on its restriction
\[
\YY/Z(\YY) \longrightarrow [\XX,\XX] \longleftrightarrow  \psl_2(\KK).
\]
We reverse the first homomorphism making it
\[
\YY/Z(\YY) \longleftrightarrow [\XX,\XX]
\]
(this step requires a careful analysis of the corresponding constructions from \cite{BY2018}) and then lift the resulting isomorphisms
\[
\YY/Z(\YY) \longleftrightarrow  \psl_2(\KK)
\]
to
\[
\YY  \longleftrightarrow  \sl_2(\KK).
\]
This is done in Section \ref{sec:sl2-proxy}

It will become clear that the appropriate fragments of this proof, together with \cite{carter1972}, provide a proof of Theorem \ref{corollary}; some additional details are given in Section \ref{sec:Corollary-2}.

\section{Orthogonal groups in two types of bases}
\label{sec:bilinear}

\subsection{Generalities on symmetric bilinear forms}

Let $V$ be a vector space of dimension $3$ over a black box field $\KK \encr \bF$, where $\bF$ is an unknown finite field of unknown odd characteristic. An important additional assumption that we are making is that we are given a computationally feasible global exponent for $\KK$, that is, a natural number $E$ such that $\bx{a}^E = \bx{1}$ for all $\bx{a} \in \KK \smallsetminus \{\bx{0}\}$, so that we can compute square roots in $\KK$, when they exist, by a version of the Tonelli-Shanks algorithm, \cite[Lemma 5.6]{BY2018}.

Assume that $\beta(\cdot, \cdot)$ is a non-degenerate symmetric bilinear form on $V$. It is well-known \cite[Section 1.4]{carter1972} that $\beta$ has Witt index $1$ and that there are only two classes of equivalence of non-degenerate symmetric bilinear forms on $V$, and if $\beta$ belongs to one of these classes, then $\bxg{\epsilon}\beta$, where $\bxg{\epsilon}$ is not a square root in $\KK$, belongs to another class.

We set $Q(v)=\beta(v,v,)$; this is the \emph{quadratic form} associated with $\beta$. (In the literature the quadratic form associated with $\beta$ is frequently taken to be $Q(v)=\frac{1}{2}\beta(v,v)$; we feel that our choice simplifies some our calculations.)

Notice that for arbitrary $\bxg{\epsilon}\in \KK$ the orthogonal groups $\so(V, \beta)$ and $\so(V, \bxg{\epsilon}\beta)$ coincide \emph{elementwise}.

It is important to keep this basic observation in mind because in the algorithms that we develop in this paper, the orthogonal  groups $\so(V, \beta)$ will be their \emph{sets of inputs}. Moreover, they will be given to us as subsets of the matrix group $\gl_3(\KK)$. Writing orthogonal transformations from $\so(V, \beta)$ in different bases of $V$ introduces some subtle changes which we will have to take into account.

\subsection{Two types of bases: spinor and canonical}

We shall call a basis $\mathcal{B} = \{\,v_1,v_2,v_3\,\}$ of $V$ a \emph{spinor} basis if
\[
\beta(v_i, v_j) = \left\{\begin{array}{ll}
                \bxg{\lambda} \mbox{ for some fixed } \bx{1} \ne \bxg{\lambda}\in \KK & \mbox{if } i=j\\
                \bx{0} & \mbox{if } i\ne j
                \end{array} \right. .
\]
In a spinor basis, the quadratic form $Q$ associated with $\beta$ is written by the scalar matrix $\bxg{\lambda}{ I}$, and we will denote the group of matrices which preserves this form as $O^{\sharp}_3(\KK)$, the corresponding special orthogonal group as $\so_3^{\sharp}(\KK)$, and its commutator subgroup as $\om_3^{\sharp}(\KK)$.

Therefore
\[
O^{\sharp}_3(\KK) = \{\, M\in \gl_3(\KK) : M^t \cdot \bxg{\lambda}{ I}\cdot M  = \bxg{\lambda}{ I} \,\},
\]
or, which is the same,
\[
O^{\sharp}_3(\KK) = \{\, M\in \gl_3(\KK) : M^t  M  = { I} \,\},
\]
which is the standard definition of the orthogonal group.

We shall call a basis $\mathcal{C} = \{\, e, w, f\,\}$ of $V$  \emph{canonical}  if the quadratic form $Q$ is written in it by the matrix $\bxg{\lambda} J$, where
\[
J = \begin{bmatrix} 0 & 0 & 1 \\ 0 & -2 &  0 \\ 1 & 0 & 0\end{bmatrix}.
\]

We define  $O^{\flat}_3(\KK)$ as
\[
O^{\flat}_3(\KK) =  \{\, M\in \gl_3(\KK) : M^t \cdot \bxg{\lambda}{ J}\cdot M  = \bxg{\lambda}{ J} \,\},
\]
or, which is the same,
\[
O^{\flat}_3(\KK) = \{\, M\in \gl_3(\KK) : M^t J  M  = { J} \,\},
\]
with $\so_3^{\flat}(\KK)$ and $\om_3^{\flat}(\KK)$ defined in an obvious way.

To summarize, the subgroups $O^{\sharp}_3(\KK)$ and $O^{\flat}_3(\KK)$ in $\gl_3(\KK)$ represent the same orthogonal group $O_3(V, \beta)$ written in two different bases, one of them is spinor, another canonical.  The groups $O^{\sharp}_3(\KK)$ and $O^{\flat}_3(\KK)$ do not change if we replace the corresponding symmetric bilinear form $\beta$ by its non-zero scalar multiple $\bxg{\lambda}\beta$.

\subsection{Change of basis}
\label{sec:change-of-basis}

So we have to find the change of basis matrix that conjugates $\so_3^{\sharp}(\KK)$ to $\so_3^{\flat}(\KK)$ in $\gl_3(\KK)$.

Let us take a canonical basis $\mathcal{C}=\{e,w,f\}$ with $\bxg{\lambda} = \bx{1}$. Then
\[
\beta(e,e) = \beta(f,f) = \beta(e,w) = \beta(f,w) = 0, \; \beta (e,f) = 1, \; \beta(w,w) = -2.
\]
In every finite field $\KK$ of odd characteristic there exist $\bx{a}, \bx{b} \in \KK$ such that $\bx{a}^2+\bx{b}^2 = -\bx{1}$, they can be easily found, in probabilistic time polynomial in $\log E$, with the help of the Tonelli-Shanks algorithm, \cite[Lemma 5.6]{BY2018}. Note that we can compute such $a$ and $b$ without knowing the characteristic of $\KK$. Then a direct calculation shows that the vectors
\bea
v_1 &=& e+f\\
v_2 &=& -\bx{b}e + \bx{a}w +  \bx{b}f\\
v_3 &=& \bx{a}e +\bx{b}w - \bx{a}f
\eea
form a spinor basis, let us call it $\mathcal{B}$, and we have the change of basis matrix from $\mathcal{C}$ to $\mathcal{B}$
\[
P=\bbm
\bx{1}&-\bx{b}&\bx{a}\\
\bx{0}&\bx{a}&\bx{b}\\
\bx{1} &\bx{b}&-\bx{a}
\ebm.
\]
If  $\KK$ contains square root of $-\bx{1}$, say, $\bxg{\epsilon}^2 = -\bx{1}$, then we can take $\bx{a}= \bxg{\epsilon}$,  $\bx{b} = \bx{0}$, and get a simpler transition matrix
\[
P=\bbm
\bx{1}&\bx{0}&\bxg{\epsilon}\\
\bx{0}&\bxg{\epsilon}&\bx{0}\\
\bx{1} &\bx{0}&-\bxg{\epsilon}
\ebm.
\]

\subsubsection*{Analysis of this calculation}

The 3-dimensional vector space $V= \KK^3$ with a non-degenerate symmetric bilinear form $\beta(\cdot,\cdot)$ has a model which is very natural in the context of this paper: the space (actually, the Lie algebra) $\lsl_2(\KK)$ of $2\times 2$ matrices over $\KK$ of trace $\bx{0}$ with
\[
\beta(U,V) = -{\rm Tr}(UV).
\]
For the space $\lsl_2(\KK)$, the matrices
\[
E = \bbm \bx{0} & \bx{0} \\ -\bx{1} & \bx{0}\ebm, \quad W = \bbm \bx{1} & \bx{0} \\ \bx{0} & -\bx{1} \ebm, \quad F = \bbm \bx{0} & \bx{1}\\ \bx{0} & \bx{0} \ebm
\]
form a canonical basis; applying construction of a spinor basis as described above, we get
\[
V_1 = \bbm \bx{0} & \bx{1}\\ -\bx{1} & \bx{0}\ebm, \quad
V_2 = \bbm \bx{a} & \bx{b} \\ \bx{b} & -\bx{a}  \ebm,
\quad V_3 = \bbm \bx{b} & -\bx{a}\\ -\bx{a} &-\bx{b} \ebm.
\]
The matrices $V_1, V_2, V_3$ are generators of three cyclic subgroups of order $4$ in a quaternion group $\QQ < \gl_2(\KK)$ and they satisfy the following relations:
\[
V_1^2 = V_2^2=V_3^2 = -1, \quad V_1V_2 = V_3,\quad V_2V_3=V_1, \quad V_3V_1 = V_2.
\]
Our previous paper \cite[Section 9]{BY2018} explains, in a pure black box and hence coordinate-free context,  why finding a quaternion subgroup $\QQ$ amounts to constructing of a spinor basis in $\lsl_2(\KK)$. In \cite[Section 8]{BY2018}, computing in a black box group $\XX \encr \pgl_2(\bF)$, we construct the image $\overline{\QQ}$ of $\QQ$ in $\XX$ and its normaliser $N_{\XX}(\overline{\QQ}) \encr \Sym_4$, and this is one of the key steps in the algorithm developed in \cite{BY2018}.

\subsection{Isomomorphisms $\pgl_2(\KK) \longleftrightarrow \so_3^{\flat}(\KK)$} \label{sec:isos}

\begin{proposition}\label{pro:om3-psl2}
Let $\KK$ be a black box field and let $E$ be a global exponent for the multiplicative group $\KK^*$. Then there is a Las Vegas algorithm  which constructs, in probabilistic time polynomial in $\log E$,  two-way isomorphism
\[
\Phi:\pgl_2(\KK) \longrightarrow \so_3^{\flat}(\KK), \quad \Phi^{-1}: \so_3^{\flat}(\KK)\longrightarrow \pgl_2(\KK).
\]
The algorithm runs in time polynomial in $\log E$.
\end{proposition}

\begin{proof} The required isomorphism comes from the action of $\gl_2(\KK)$ on the Lie algebra $\mathfrak{l} = \lsl_2(\KK)$ of $2 \times 2$ matrices over $\KK$ of trace $\bx{0}$. Following Section \ref{sec:change-of-basis}, we choose a canonical basis in $\mathfrak{l}$ as
\[
E = \bbm \bx{0} & \bx{0} \\ -\bx{1} & \bx{0}\ebm, \quad W = \bbm \bx{1} & \bx{0} \\ \bx{0} & -\bx{1} \ebm, \quad F = \bbm \bx{0} & \bx{1}\\ \bx{0} & \bx{0} \ebm.
\]
Let
\[
A =  \bbm \bx{a} & \bx{b} \\ \bx{c} & \bx{d}\ebm \in \gl_2(\KK),
\]
then
\[
A^{-1} = \bbm \bx{a} & \bx{b} \\ \bx{c} & \bx{d}\ebm^{-1} = \frac{\bx{1}}{\bx{ad-bc}}\bbm \bx{d} & \bx{-b} \\ \bx{-c} & \bx{a}\ebm,
\]
and it is easy to compute that
\bea
E^{A} &=& \frac{\bx{1}}{\bx{ad-bc}}\bbm \bx{ab} & \bx{b}^2 \\ -\bx{a}^2 & \bx{-bd}\ebm\\
&=& \frac{\bx{1}}{\bx{ad-bc}} \left( \bx{a}^2E +\bx{ab}\,W +\bx{b}^2F\right),\\
W^{A} &=& \frac{\bx{1}}{\bx{ad-bc}}\bbm \bx{ad + bc} & \bx{2bd} \\ -\bx{2ac}^2 & \bx{-ad -bc}\ebm\\
&=& \frac{\bx{1}}{\bx{ad-bc}} \left( \bx{2ac}E +(\bx{ad+bc})\,W +\bx{2bd}F\right),\\
F^{A} &=& \frac{\bx{1}}{\bx{ad-bc}}\bbm \bx{cd} & \bx{d}^2 \\ -\bx{c}^2 & \bx{-cd}\ebm\\
&=& \frac{\bx{1}}{\bx{ad-bc}} \left( \bx{c}^2E +\bx{cd}\,W +\bx{d}^2F\right).
\eea
Therefore the conjugation by the matrix $A=\bbm \bx{a} & \bx{b} \\ \bx{c} & \bx{d}\ebm$ is written in the basis $E$, $W$, $F$ by the matrix
\[
\frac{\bx{1}}{\bx{ad-bc}}\bbm
\bx{a}^2 & \bx{2ac}   & \bx{c}^2\\
\bx{ab}  & \bx{ad+bc} & \bx{cd}\\
\bx{b^2} & \bx{2bd}   & \bx{d}^2
\ebm,
\]
and we have a homomorphism
\begin{equation} \label{eq:Phi}
A = \bbm \bx{a} & \bx{b} \\ \bx{c} & \bx{d}\ebm \mapsto
\bbm
\bx{a}^2\bxg{\delta} & \bx{2ac}\bxg{\delta}   & \bx{c}^2\bxg{\delta}\\
\bx{ab}\bxg{\delta}  & (\bx{ad+bc})\bxg{\delta} & \bx{cd}\bxg{\delta}\\
\bx{b^2}\bxg{\delta} & \bx{2bd}\bxg{\delta}   & \bx{d}^2\bxg{\delta}
\ebm, \quad \bxg{\delta} = \frac{\bx{1}}{\bx{ad-bc}},
\end{equation}
from $\gl_2(\KK)$ to $\so_3^\flat(\KK)$. It is easy to check that the kernel of this homomorphism is the group of scalar matrices and results in an isomorphism
\[
\Phi: \pgl_2(\KK) \longrightarrow \so_3^\flat(\KK).
\]

The inverse isomomorphism $$\Phi^{-1} :\so_3^{\flat}(\KK) \longrightarrow \pgl_2(\KK)$$ can now be found with ease. Note that, by \cite[Lemma 5.6]{BY2018}, we can construct, if they exist, the square roots of the elements of the black box field $\KK$ in polynomial time.

Assume that we are given a matrix
\[
B = \bbm
\bx{b}_{11} & \bx{b}_{12} & \bx{b}_{13} \\
\bx{b}_{21} & \bx{b}_{22} & \bx{b}_{23} \\
\bx{b}_{31} & \bx{b}_{32} & \bx{b}_{33}
\ebm \in \so_3^\flat(\KK)
\]
and wish to find $A =\bbm \bx{a} & \bx{b} \\ \bx{c} & \bx{d}\ebm  $ such that $\Phi(A) = B$. Because of Equation (\ref{eq:Phi}), this amounts to solving the system of equations in variables $\bx{a},\, \bx{b},\, \bx{c},\, \bx{d}$
\begin{equation} \label{eq:system}
\bbm
\bx{a}^2\bxg{\delta} & \bx{2ac}\bxg{\delta}   & \bx{c}^2\bxg{\delta}\\
\bx{ab}\bxg{\delta}  & (\bx{ad+bc})\bxg{\delta} & \bx{cd}\bxg{\delta}\\
\bx{b^2}\bxg{\delta} & \bx{2bd}\bxg{\delta}   & \bx{d}^2\bxg{\delta}
\ebm = \bbm
\bx{b}_{11} & \bx{b}_{12} & \bx{b}_{13} \\
\bx{b}_{21} & \bx{b}_{22} & \bx{b}_{23} \\
\bx{b}_{31} & \bx{b}_{32} & \bx{b}_{33}
\ebm
\end{equation}
It is easy to see that at least one of the matrix elements $\bx{b}_{11},\, \bx{b}_{13},\, \bx{b}_{31},\, \bx{b}_{33}$ is not zero; assume that $\bx{b}_{11} \ne \bx{0}$, other cases can be treated similarly.

If $\bx{b}_{11}= \bx{a}^2\bxg{\delta}$ has no square root in $\KK$, then $\bxg{\delta}$ also has no square root.  In that case, pick some $\bxg{\gamma}$ which is not a square root in $\KK$; alternatively set $\bxg{\gamma} = \bx{1}$. In the both cases $\bxg{\delta}\bxg{\gamma}$ is a square and, for the sake of argument, \emph{denote} (but do not compute -- we cannot compute because we do not know $\bxg{\delta}$)
\[
\bxg{\epsilon}^2 = \bxg{\delta}\bxg{\gamma},
\]
and \emph{compute}
\[
\bx{b}'_{ij} = \bx{b}_{ij}\bxg{\gamma} \; \mbox{ for } \; i, j = 1,2,3.
\]
This allows us to rewrite Equation (\ref{eq:system}) as
\begin{equation} \label{eq:rearranged}
\bbm
\bx{a}^2\bxg{\epsilon}^2 & \bx{2ac}\bxg{\epsilon}^2   & \bx{c}^2\bxg{\epsilon}^2\\
\bx{ab}\bxg{\epsilon}^2  & (\bx{ad+bc})\bxg{\epsilon}^2 & \bx{cd}\bxg{\epsilon}^2\\
\bx{b^2}\bxg{\epsilon}^2 & \bx{2bd}\bxg{\epsilon}^2   & \bx{d}^2\bxg{\epsilon}^2
\ebm = \bbm
\bx{b}'_{11} & \bx{b}'_{12} & \bx{b}'_{13} \\
\bx{b}'_{21} & \bx{b}'_{22} & \bx{b}'_{23} \\
\bx{b}'_{31} & \bx{b}'_{32} & \bx{b}'_{33},
\ebm
\end{equation}
which can be immediately solved:
\bea
 \sqrt\bx{{b}'_{11}} &=& \bx{a}\bxg{\epsilon}\\
 \frac{\bx{b}'_{21}}{\sqrt\bx{{b}'_{11}}} &=& \bx{b}\bxg{\epsilon}\\
   \frac{\bx{b}'_{12}}{2\sqrt\bx{{b}'_{11}}} &=& \bx{c}\bxg{\epsilon}\\
   \frac{2\bx{b}'_{23}\cdot \sqrt\bx{{b}'_{11}}}{\bx{b}'_{12}}  &=& \bx{d}\bxg{\epsilon},
\eea
which yields us the matrix
\[
\bxg{\epsilon}\bbm \bx{a} & \bx{b} \\ \bx{c} & \bx{d}\ebm = \bxg{\epsilon}A
\]
which is the same element of $\pgl_2(\KK)$ as $A$. Notice that we do not compute $\bxg{\delta}$ and $\bxg{\epsilon}$.

This establishes the isomomorphism $$\Phi^{-1} :\so_3^{\flat}(\KK) \longrightarrow \pgl_2(\KK)$$
\end{proof}

\section{Construction of a proxy for $\YY\encr \sl_2$}\label{sec:sl2-proxy}

In this section, we present the proof of Theorem \ref{theo:proxy-sl2}. The following lemma is crucial.

\begin{lemma}\label{lemma:existence_and_uniquiness}
Let $X$ and $Y$ be two groups isomorphic to $\sl_2(\bF)$ over a finite field\/ $\bF$ of odd characteristic, then any surjective homomorphism
\[
X \longrightarrow Y/Z(Y)
\]
can be lifted to a homomorphism
\[
X \longrightarrow Y,
\]
and this homomorphism is unique.
\end{lemma}

\begin{proof}
The proof immediately follows from the well-known property: every automorphism of $Y/Z(Y)$ can be lifted to an automorphism of $Y$, and this automorphism of $Y$ is unique.
\end{proof}

To prove Theorem \ref{theo:proxy-sl2}, we need some details of the constructions from \cite{BY2018} which we shall give a summary here.  Let $\YY \encr \sl_2(\bF)$. We construct two cyclic subgroups $\SS$ (torus of order twice odd number) and $\RR$ (torus containing an element of order 4) in $\YY$ and form the direct product $\YY \times \YY$. Then, we consider the black box subgroup $\XX^*$ which is generated in $\YY \times \YY$ by the pairs $(\bx{s}, \bx{s})$ for $\bx{s} \in \SS$ and  $(\bx{r}, \bx{r}^{-1})$ for $\bx{r} \in \RR$. Now, $\XX = \XX^*/Z(\XX^*)\langle \bxg{\delta}\rangle\encr \pgl_2(\bF)$, where $\bxg{\delta}$ is the involution swapping the two copies of $\YY$ in $\YY\times \YY$. By the results in \cite{BY2018}, these constructions lead up to a construction of a black box field $\KK\encr \bF$ and the morphisms
\[
\XX \longleftrightarrow \so_3^\sharp(\KK).
\]

\subsection{Construction of the morphism $\sl_2(\KK)\rightarrow \YY$}\label{sec:morp-sl2-Y} By the construction of the morphisms

\[
\XX \longleftrightarrow \so_3^\sharp(\KK)
\]
from \cite{BY2018} and also
\[
\so_3^\flat(\KK) \longleftrightarrow  \pgl_2(\KK)
\]
in Section \ref{sec:isos} together with
\[
\so_3^\sharp(\KK) \longleftrightarrow \so_3^\flat(\KK)
\]
in Section \ref{sec:change-of-basis} we have a chain of morphisms
\[
\XX \longleftrightarrow \so_3^\sharp(\KK) \longleftrightarrow \so_3^\flat(\KK) \longleftrightarrow  \pgl_2(\KK).\]
Reading this diagram from the right to the left and restricting the map to $\psl_2(\KK)$, we can get a chain of morphisms
\[
 \psl_2(\KK) \longrightarrow \Omega_3^\sharp(\KK) \longrightarrow \XX^*/Z(\XX^*),
\]
then we expand it to
\[
\sl_2(\KK) \longrightarrow \psl_2(\KK) \longrightarrow \Omega_3^\sharp(\KK) \longrightarrow \XX^*/Z(\XX^*) \longrightarrow \YY/Z(\YY),
\]
where the last arrow is induced by the natural projection of $\XX^* = \YY \times \YY$ on its direct factor. Hence, we have a morphism
\[
\bxg{\phi}: \sl_2(\KK) \longrightarrow \YY/Z(\YY).
\]

Now, we shall lift this morphism $\bxg{\phi}$ to the desired morphism
\[
\bxg{\psi}: \sl_2(\KK) \longrightarrow \YY.
\]

Let $\bx{z} \in Z(\YY)$ be the central involution of $\YY$. If $\bx{x} \in\sl_2(\KK)$ then $\bxg{\phi}(\bx{x})$ is a coset in $\YY$ made of two elements, say $ \bx{y}$ and $\bx{yz}$. If $\bx{x}$ is of odd order then one of the elements in the coset $\{\bx{y},\bx{yz}\}$ has odd order, and, by Lemma \ref{lemma:existence_and_uniquiness}, is equal to the image $\bxg{\psi}(\bx{x})$ of $\bx{x}$.

It is well-known that every matrix $\bx{x} \in \sl_2(\KK)$ can be written as a product of $k\leqslant 4$  transvections, $\bx{x} = \bx{x}_1 \cdots \bx{x}_k$; an explicit formulae are in \cite[pp.~81--82]{carter1972}. Indeed,
if
\[
\bbm \bx{a} & \bx{b} \\ \bx{c} & \bx{d} \ebm \in \sl_2(\KK)
\]
and $\bx{c} \ne \bx{0}$, then
\[
\bbm \bx{a} & \bx{b} \\ \bx{c} & \bx{d} \ebm  = \bbm \bx{1} & (\bx{a}-\bx{1})\bx{c}^{-1} \\ \bx{0} & \bx{1} \ebm
    \bbm \bx{1} & \bx{0} \\ \bx{c} & \bx{1} \ebm
    \bbm \bx{1} & (\bx{d} - \bx{1})\bx{c}^{-1} \\ \bx{0} & \bx{1} \ebm.
\]
If $\bx{b} \ne \bx{0}$ we have
\[
\bbm \bx{a} & \bx{b} \\ \bx{c} & \bx{d} \ebm  = \bbm \bx{1} & \bx{0} \\ (\bx{d} - \bx{1})\bx{b}^{-1} & \bx{1} \ebm
    \bbm \bx{1} & \bx{b} \\ \bx{0} & \bx{1} \ebm
    \bbm \bx{1} & \bx{0} \\ (\bx{a} - \bx{1})\bx{b}^{-1} & \bx{1} \ebm.
\]
If $\bx{b} = \bx{c} = \bx{0}$ we have
\[
\bbm \bx{a} & \bx{0} \\ \bx{0} & \bx{a}^{-1} \ebm =  \bbm \bx{1} & \bx{0} \\ \bx{a}^{-1} -\bx{1} & \bx{1} \ebm
        \bbm \bx{1} & \bx{1} \\ \bx{0} & \bx{1} \ebm
        \bbm \bx{1} & \bx{0} \\ \bx{a}-\bx{1} & \bx{1} \ebm
        \bbm \bx{1} & \bx{-a}^{-1} \\ \bx{0} & \bx{1} \ebm.
\]
Since we work in fields of odd characteristic, transvections are elements of odd order, and the previous argument allows us to compute $\bxg{\psi}(\bx{x})$ as
\[
\bxg{\psi}(\bx{x}) = \bxg{\psi}(\bx{x}_1) \cdots \bxg{\psi}(\bx{x}_k).
\]

\subsection{Construction of the morphism $\YY \rightarrow \sl_2(\KK)$}\label{sec:morp-Y-sl2}
To construct the reverse morphism presented in Section \ref{sec:morp-sl2-Y}, it is important to observe that the morphism
\[
\bxg{\psi}:\sl_2(\KK) \longrightarrow \YY
\]
is reversible on $\SS$ and $\RR$ since we have natural maps
\[
\SS \longrightarrow \XX^* \mbox{ and } \RR \longrightarrow \XX^*
\]
and we can map them back to $\sl_2(\KK)$.

Now we show how to reverse $\bxg{\psi}$ on the entire $\YY$. Let us denote $\bar{\YY} = \YY/Z(\YY)$.  Abusing notation, we may use the same notation for elements in $\bar{\YY} \encr \psl_2$ as for elements in $\YY\encr \sl_2$.

Indeed it will suffice to reverse the  map $\bxg{\rho}$ induced by $\bxg{\psi}$  on
\[
\bxg{\rho}: \psl_2(\KK) \longrightarrow \bar{\YY},
\]
and have a morphism
\[
\bxg{\rho}^{-1}: \bar{\YY}\longrightarrow \psl_2(\KK),
\]
expand it to
\[
\bxg{\sigma}: \YY \longrightarrow \psl_2(\KK)
\]
and then lift it to a map
\[
\bxg{\theta}: \YY \longrightarrow \sl_2(\KK).
\]

Let us call elements in $\bar{\YY}$ with already known preimages in $\psl_2(\KK)$ \emph{``white''}. Obviously, products of white elements are white.

We shall prove that every element in $\YY$ is white.

 \begin{lemma} All elements in
 $N_{\bar{\YY}}(\SS)$ and $N_{\bar{\YY}}(\RR)$ are white.
 \end{lemma}

 \begin{proof} It suffices to prove the statement for involutions in $N_{\bar{\YY}}(\SS)\smallsetminus \SS$. We first construct one such involution. Since the elements of $\SS$ are white, we can represent any element from $\SS$ by $2 \times 2$ matrices with entries from $\KK$. Let $\bx{s} \in \SS$ be an element of order bigger than or equal to 3 and $M$ be its image in $\sl_2(\KK)$. Then we need to locate an involution $A\in \sl_2(\KK)$ satisfying $M^A=M^{-1}$ which is equivalent to $MA=AM^{-1}$. The entries of such a matrix $A$ can be found by solving a system of linear equations over the black box field $\KK$. Now, by using the map $\bxg{\psi}$ from Subsection \ref{sec:morp-sl2-Y}, we construct a white element $\bx{u}=\bxg{\psi}(A) \in N_{\bar{\YY}}(\SS)\smallsetminus \SS$. Now, any other involution $\bx{t} \in N_{\bar{\YY}}(\SS)$ can be written as $\bx{t}=\bx{u}\cdot \bx{u}\bx{t}$, with  $\bx{u}\bx{t} \in \SS$ being a white element.
 \end{proof}

 \begin{lemma}
 If $\bx{a}$ is a white involution then all elements in $C_{\bar{\YY}}(\bx{a})$ are white.
  \end{lemma}

  \begin{proof}
  One of the white tori $\SS$ or $\RR$ contains an involution; without loss of generality we can assume that this is $\bx{s} \in \SS$. Being white involutions, $\bx{a}$ and $\bx{s}$ are conjugate by a white element (we can do the corresponding calculation in $\psl_2(\KK)$), hence  $C_{\bar{\YY}}(\bx{a})$ is conjugate to the white subgroup $C_{\bar{\YY}}(\bx{s}) = N_{\bar{\YY}}(\SS)$ by a white element and is therefore white.
  \end{proof}

We can now complete construction of $\bxg{\rho}^{-1}: \bar{\YY} \longrightarrow \psl_2(\KK)$.

\begin{lemma}
Every involution in $\bar{\YY}$ is white.
\end{lemma}

\begin{proof}
 Let $\bx{t}\in \bar{\YY}$ be an involution. Taking random white involutions (that is, images of random involutions from $\psl_2(\KK)$), we can find  a white involution $\bx{a}$ such that the product $\bx{a}\bx{t}$ is of even order, thus yielding an involution $\bx{z}$ commuting with both $\bx{a}$ and $\bx{t}$; this involution $\bx{z}$ is therefore white.
 This means that we can produce random white involutions in $C_{\bar{\YY}}(\bx{t})$  until they generate a white dihedral subgroup containing $\bx{t}$.
\end{proof}

\begin{lemma}
Every element of $\bar{\YY}$ is white.
\end{lemma}

\begin{proof}
Applying the same arguments in \cite[Lemma 5.4]{BY2018}, we have a Las Vegas polynomial time algorithm with which we can write every element of $\bar{\YY}$ as a product of  involutions. Since every involution is white, every element is white.
\end{proof}

We can now complete the proof of Theorem \ref{theo:proxy-sl2}. Indeed we have the inverse morphism $\bxg{\rho}^{-1}$ and hence we have a morphism
 \[
\bxg{\sigma}: \YY \longrightarrow \psl_2(\KK).
\]
Let $\bx{y}\in \YY$. We can compute $\bxg{\sigma}(\bx{y})$ as the coset in $\psl_2(\KK)$ consisting of two elements $\bx{u}$ and $\bx{v}$, and compute $\bxg{\rho}(\bx{u})$.
If $\bxg{\rho}(\bx{u}) = \bx{y}$, then $\bxg{\theta}(\bx{y}) = \bx{u}$, otherwise $\bxg{\theta}(\bx{y}) = \bx{v}$.

\section{Proof of Theorem \ref{corollary}}
\label{sec:Corollary-2}

In case of $\XX \encr \pgl_2(\bF)$, a proof of Theorem \ref{corollary} is a simple combination of Remark \ref{rem:so3} and Proposition \ref{pro:om3-psl2}.

In case of $\XX \encr \psl_2(\bF)$ the proof is a slight modification of arguments of Section \ref{sec:sl2-proxy}.

\section*{Acknowledgements}

This paper---and other papers in our project---would have never been written if the authors did not enjoy the warm hospitality offered to them at the Nesin Mathematics Village in \c{S}irince, Izmir Province, Turkey, in  2018--20 as part of their Research in Pairs programme; our thanks go to Ali Nesin and to all volunteers, staff, and students who have made the Village a mathematical paradise. 

Our work was partially supported by CoDiMa (CCP in the area of Computational Discrete Mathematics; EPSRC grant 	EP/M022641/1).

In the project, we were using the GAP software package by The GAP Group, GAP--Groups, Algorithms, and Programming, Version 4.8.7; 2017 (\url{http://www.gap-system.org}).


\end{document}